\documentclass[11pt]{article}
\usepackage{epsfig}
\usepackage[latin1]{inputenc}
\usepackage{amsmath,amssymb,amsthm,amsxtra}
\usepackage{enumerate}
\usepackage{verbatim}
\usepackage{amsfonts}
\usepackage{dsfont}
\usepackage{longtable}

\usepackage{tikz}
\usetikzlibrary{decorations.pathreplacing, calligraphy}

\usepackage{hyperref}

\usepackage{enumitem}
\usepackage{cite}
\usepackage{geometry}
\usepackage{graphicx}
\usepackage{amssymb}
\usepackage{amsopn}
\usepackage{multirow}
\usepackage{multicol}
\usepackage{epstopdf}
\usepackage{setspace}
\usepackage{cite}

\newtheorem{thm}{Theorem}[section]
\newtheorem{cor}[thm]{Corollary}
\newtheorem{lem}[thm]{Lemma}
\newtheorem{prop}[thm]{Proposition}
\newtheorem{defn}[thm]{Definition}

\newcommand{\RR}{\mathbb{R}}
\newcommand{\CC}{\mathbb{C}}
\newcommand{\ZZ}{\mathbb{Z}}
\newcommand{\QQ}{\mathbb{Q}}
\newcommand{\NN}{\mathbb{N}}

\newcommand{\PP}{\mathbb{P}}
\newcommand{\FF}{\mathbb{F}}
\newcommand{\HH}{\mathbb{H}}

\newcommand{\cc}{\mathcal{C}}
\newcommand{\dc}{\mathcal{D}}
\newcommand{\ec}{\mathcal{E}}
\newcommand{\fc}{\mathcal{F}}

\newcommand{\lc}{\mathcal{L}}
\newcommand{\mc}{\mathcal{M}}

\newcommand{\qc}{\mathcal{Q}}

\newcommand{\la}{\langle}
\newcommand{\ra}{\rangle}
\newcommand{\bsm}{\left ( \begin{smallmatrix} }
\newcommand{\esm}{\end{smallmatrix} \right )}
\newcommand{\bma}{\begin{pmatrix}}
\newcommand{\ema}{\end{pmatrix}}
\newcommand{\opn}{\operatorname}
\newcommand{\op}{\oplus}
\newcommand{\beq}{\begin{equation*}}
\newcommand{\eeq}{\end{equation*}}
\newcommand{\bpm}{\begin{pmatrix}}
\newcommand{\epm}{\end{pmatrix}}
\newcommand{\bi}{\begin{itemize}}
\newcommand{\ei}{\end{itemize}}

\title{The Baily-Borel compactification of a family of orthogonal modular varieties}
\author{Matthew Dawes}
\date{}
\begin{document}
\maketitle
\def\thefootnote{}
\begin{abstract}
  We study the Baily-Borel compactification of a family of four-dimensional orthogonal modular varieties
  arising as period spaces of compact hyperk\"ahler manifolds of deformation generalised Kummer type.
  Our main results concern the classification of boundary components, their incidence relations and combinatorics.
\end{abstract}
\section{Introduction}
An \emph{orthogonal modular variety} is a locally symmetric variety
given by the quotient of a Hermitian symmetric space of type IV by a discrete subgroup of the orthogonal group
$\opn{O}(2, n)$.
The purpose of this paper is to study a family of 4-dimensional
orthogonal modular varieties (defined in \S\ref{famiglia}) related to moduli and periods of   compact
hyperk\"ahler manifolds of deformation generalised Kummer
type (\emph{deformation generalised Kummer varieties}).
Our main results concern the geometry and  combinatorics of the Baily-Borel compactification:
we describe the isomorphism types of boundary components (Theorem \ref{curvethm}), their incidence relations (Theorem \ref{L2boundarythm} and \ref{2p2curvethm}) and combinatorics (Corollary \ref{curvecountcor}).
We believe these are the first such results for orthogonal modular varieties of dimension 4, complementing results in dimension 10 and 19 for moduli spaces of Enriques and K3 surfaces, respectively \cite{Sterk, Scattone}.
\subsection{Lattices}
A \emph{lattice} $L$ is an even, integral quadratic form on a free abelian group of finite rank.
Unless otherwise stated, we will assume that all lattices are non-degenerate.
By Sylvester's law of inertia, the quadratic form on $L \otimes \RR$ can be diagonalised and the pair consisting of the number of positive and negative terms in the diagonalisation is known as the \emph{signature} of the lattice.
We will use $x^2$ to denote the quadratic form of $L$ evaluated at $x \in L$ and $(x,y)$ to denote the bilinear form of $L$ evaluated at $x,y \in L$ (we will also extend this convention to $L \otimes \QQ$ and $L \otimes \RR$).
Examples of lattices include the rank 1 lattice $\la d \ra$ generated by a single element $x \in L$ of length $x^2 = d$; the root lattice $A_2$; and the hyperbolic plane $U$, whose Gram matrix is given by
\beq
\bpm 0 & 1 \\ 1 & 0 \epm
\eeq
on a suitable basis (the \emph{canonical} basis).
We use $L_1 \op L_2$ to denote the  orthogonal direct sum of lattices $L_1$ and $L_2$; and $nL_1$ to denote the  orthogonal direct sum of $n$ copies of $L_1$.
We let $L(m)$ denote the lattice obtained by multiplying the quadratic form of $L$ by $m$.
If $S \subset L$ is a sublattice, we let $S^{\perp} \subset L$ denote the orthogonal complement of $S$ in $L$.

The \emph{dual lattice} $L^{\vee}$ of $L$ is the free abelian group $\opn{Hom}(L, \ZZ) \subset L \otimes \QQ$ with a  quadratic form inherited from $L$.
The quotient $D(L):=L^{\vee}/L$ (known as the \emph{discriminant group} of $L$)  inherits both a $\QQ/2\ZZ$-valued quadratic form $q_L$ (the \emph{discriminant form} of $L$) and a $\QQ/\ZZ$-valued bilinear form $b_L$ from $L$ \cite{Nikulin}.
We will often encode $b_L$ by the data $(B, \bigoplus_j C_{i_j})$ where $D(L) \cong \bigoplus_j C_{i_j}$, $C_i$ is the cyclic group of order $i$ and $B$ is the Gram matrix of $b_L$ on a canonical basis of $\bigoplus_j C_{i_j}$.

A (possibly degenerate) sublattice $S \subset L$ is said to be \emph{primitive} if $L/S$ is torsion-free and \emph{totally isotropic} if the restriction of the quadratic form from $L$ to $S$ is identically zero.
A non-zero vector $x \in  L$ is said to be \emph{primitive} (or \emph{isotropic}) if it defines a primitive (or totally isotropic) lattice $\la x \ra$.
We define the \emph{divisor} $\opn{div}(x)$ of $0 \neq x \in L$  as the positive generator of the ideal $(x, L)$.
We note that if $0 \neq x \in L$ is primitive and $x^*:=x/\opn{div}(x) \in L^{\vee}$ then $x^* \bmod {L}$ is of order  $\opn{div}(x)$ in $D(L)$.
\subsection{The orthogonal group and spinor norm}
For a lattice $L$, we let $\opn{O}(L)$ and $\opn{O}(L \otimes \RR)$ denote the orthogonal groups of $L$ and $L \otimes \RR$, respectively.
As explained in  \cite{cassels}, every $g \in \opn{O}(L \otimes \RR)$ can be written as a product
\begin{equation}\label{refdec}
  g = \sigma_{w_1} \ldots \sigma_{w_m}
\end{equation}
where
\beq
\sigma_w: x \mapsto x - \frac{2(x,w)}{(w,w)}w \in \opn{O}(L \otimes \RR)
\eeq
is the reflection defined by $w \in L \otimes \RR$.
If $g$ is as in \eqref{refdec} then the \emph{spinor norm} $\opn{sn}_{\RR}(g)$ of $g$ is defined by \cite{Kneser}
\beq
\opn{sn}_{\RR}(g) =
  \left ( \frac{-(w_1, w_1)}{2} \right )
  \ldots
    \left ( \frac{-(w_m, w_m)}{2} \right )
\in \RR/(\RR^*)^2.
\eeq
We let $\opn{O}^+(L \otimes \RR)$ denote the kernel of the spinor norm on $\opn{O}(L \otimes \RR)$ and, for  $\Gamma \subset \opn{O}(L \otimes \RR)$, we use $\Gamma^+$ to denote the intersection $\Gamma \cap \opn{O}^+(L \otimes \RR)$.
\subsection{The stable orthogonal group}
There is a natural map
\begin{equation}\label{OODL}
\opn{O}(L) \rightarrow \opn{O}(D(L)),
\end{equation}
where $\opn{O}(D(L))$ is the subgroup of $\opn{Aut}(D(L))$ preserving $q_L$.
We let $\overline{g}$ denote the image of $g \in \opn{O}(L)$ under \eqref{OODL} and use $\widetilde{\opn{O}}(L)$ to denote the kernel of \eqref{OODL}.
For $\Gamma \subset \opn{O}(L)$, we use $\widetilde{\Gamma}$ to denote the intersection $\Gamma \cap \widetilde{\opn{O}}(L)$. 
The kernel  $\widetilde{\opn{O}}(L)$ (often referred to as the \emph{stable orthogonal group}) has the useful property that  $\widetilde{\opn{O}}(S) \subset \widetilde{\opn{O}}(L)$ for any sublattice $S \subset L$, where $g \in \widetilde{\opn{O}}(S) \cap \widetilde{\opn{O}}(L)$ acts as the identity on $S^{\perp} \subset L$ (Lemma 7.1 of \cite{handbook}).
\subsection{The Eichler criterion}
We will often need to determine orbits of vectors in lattices.
If $L$ is a lattice containing a copy of $2U$ and $v_1, v_2 \in L$ are primitive then the \emph{Eichler criterion} \cite{handbook} states that  $g v_1 = v_2$ for some $g \in \widetilde{\opn{SO}}^+(L)$ if and only if $v_1^2 = v_2^2$ and $v_1^* \equiv v_2^* \bmod{ L}$.
\subsection{Orthogonal modular varieties}
Let $L$ be a lattice of signature $(2,n)$ and let $\Gamma \subset \opn{O}^+(L \otimes \RR)$ be an arithmetic subgroup.
If $\dc_L$ is the component of
\beq
\Omega_L := \{ [x] \in \PP(L \otimes \CC) \mid (x,x)=0, (x, \overline{x})>0 \}
\eeq
preserved by $\opn{O}^+(L \otimes \RR)$, then the quotient
\beq
\fc_L(\Gamma) := \dc_L / \Gamma
\eeq
is a locally symmetric variety known as an \emph{orthogonal modular variety}.
Orthogonal modular varieties are complex analytic spaces (indeed, are even quasi-projective \cite{bailyborel}) but are typically non-compact.
\subsection{The Baily-Borel compactification}
The \emph{Baily-Borel compactification} $\fc_L(\Gamma)^*$ of  $\fc_L(\Gamma)$ is an irreducible normal complex projective variety containing $\fc_L(\Gamma)$ as a Zariski-open subset.
It is defined by $\opn{Proj} M_*(\Gamma, \mathds{1})$ where $M_*(\Gamma, \mathds{1})$  is the ring of modular forms with trivial character for $\Gamma$ \cite{bailyborel}.
In most of the paper we are interested in studying the boundary of $\fc_L(\Gamma)^*$, which can be described by Theorem \ref{bbdec}.
\begin{thm}{(\cite{handbook})}\label{bbdec}
  The Baily-Borel compactification $\fc_L(\Gamma)^*$  decomposes  as
  \beq
  \fc_L(\Gamma)^* =
  \fc_L(\Gamma)
  \sqcup \bigsqcup_{E \in \ec} \cc_E
  \sqcup \bigsqcup_{l \in \ell} P_l
  \eeq
  where $\ell$ and $\ec$ are sets of the finitely many $\Gamma$-orbits of primitive totally isotropic sublattices of rank 1 and 2 in $L$, respectively; and  the indices $E \in \ec$ and $l \in \ell$ run over a choice of representative for each orbit.
  Each $\cc_E$ is a modular curve and each $P_l$ is a point.
  The point $P_l$ is contained in the closure of $\cc_E$ if and only if representatives can be chosen such that $l \subset E$.
\end{thm}
Furthermore, if $\overline{\dc}_L$ is the topological closure of $\dc_L$ in the compact dual $\dc_L^{\vee}$, then the boundary curve $\cc_E$ is isomorphic to $\HH^+/G(E)$ where  $G(E) := \opn{Stab}_{\Gamma}(E) / \opn{Fix}_{\Gamma}(E)$ and the upper half-plane $\HH^+$ is identified with $\HH^+ \cong \PP(E \otimes \CC) \cap \overline{\dc}_L$.
\subsection{A family of orthogonal modular varieties}\label{famiglia}
From now on, we let $L_{2d}$ denote the lattice
\beq
L_{2d} = 2U \op \la -2d \ra \op \la -6 \ra
\eeq
and let $\underline{v}$ and $\underline{w}$ denote generators for the $\la -2d \ra$ and $\la -6 \ra$ factors of $L_{2d}$, respectively.
We define the group $\Gamma_{2d}$ by
\beq
\Gamma_{2d} = \{ g \in \opn{O}^+(L) \mid g \underline{v}^* \equiv \underline{v}^* \bmod{ L} \}.
\eeq
We will mostly be interested in studying the case of $d=p^2$ for prime $p>3$, where $\fc_{L_{2p^2}}(\Gamma_{2p^2})$ has particularly agreeable geometric and combinatorial properties.
\subsection{Moduli of deformation generalised Kummer varieties}
Let $A$ be an abelian surface and let $A^{[n+1]}$ be the Hilbert scheme parametrising $(n+1)$-points on $A$.
As $A^{[n+1]}$ inherits an addition from $A$, there is a natural projection
\beq
p:A^{[n+1]} \rightarrow A.
\eeq
The fibre $X:=p^{-1}(0)$ is known as a \emph{generalised Kummer variety} and is a simply connected compact K\"ahler manifold such that $H^0(X, \Omega_X^2)$ is generated by an everywhere non-degenerate holomorphic 2-form (i.e. a \emph{compact hyperk\"ahler} or \emph{irreducible symplectic} manifold) \cite{beauvillesexamples}.
Deformations of $X$ are irreducible symplectic manifolds known as \emph{deformation generalised Kummer varieties}.
The Beauville-Bogomolov-Fujiki form \cite{beauvillesexamples} endows $H^2(X, \ZZ)$ with the structure of a lattice $M$ which, by the results of Rapagnetta \cite{rapagnetta2},
is given by
\begin{equation}\label{bblattice}
M \cong 3U \op \la -2(n+1) \ra.
\end{equation}
The choice of an ample line bundle $\lc \in \opn{Pic}(X)$ defines a \emph{polarisation} for $X$.
By taking the first Chern class of $\lc$, we obtain a vector $h:=c_1(\lc) \in M$ and a lattice  $L:=h^{\perp} \subset M$.
We define the \emph{degree} $2d$ of $\lc$  by $2d:=h^2$ and the \emph{polarisation type} of $\lc$  as the $\opn{O}(M)$-orbit of $h$.
We will assume  all polarisations are \emph{primitive},  (i.e the vector $h$ is primitive in $M$) and we will only consider  \emph{split} polarisations, which are those satisfying $\opn{div}(h)=1$.
\begin{lem}
  If $h \in M$ corresponds to a split polarisation $\lc$ of degree $2d$ then, 
  \begin{enumerate}
  \item the polarisation type of $\lc$ is uniquely determined by the length $h^2$;
  \item the lattice $L \cong 2U \op \la -2(n+1) \ra \op \la -2d \ra$.
  \end{enumerate}
\end{lem}
\begin{proof}
  Apply the Eichler criterion.
\end{proof}
A full classification of non-split polarisation types can be obtained  as for irreducible symplectic manifolds of $K3^{[n]}$-type (Proposition 3.6 of \cite{ModuliSpacesOfIrreducibleSymplecticManifolds}).

By the results of Viehweg \cite{Viehweg}, there exists a GIT moduli space $\mc$ parametrising deformation generalised Kummer varieties of fixed dimension and polarisation type $\opn{O}(M). h$ \cite{handbook}.
If $\opn{O}(M, h) \subset \opn{O}(M)$ is the group defined by
\beq
\opn{O}(M, h) = \{g \in \opn{O}(M) \mid gh  = h \}
\eeq
(which need not coincide with Markman's monodromy group \cite{monok3ii,  monok3, intk3,  Mongardi}) then, by Theorem 3.8 of \cite{handbook},  there exists a  finite-to-one dominant morphism
\beq
\psi:\mc' \rightarrow \fc_L(\opn{O}^+(M, h))
\eeq
for every connected component $\mc'$ of $\mc$.
In Proposition \ref{modgprop}, we will show that when $n=2$ (corresponding to deformation generalised Kummer varieties of dimension 4) and $h$ corresponds to a  split polarisation of degree $2d$, then $\fc_L(\opn{O}^+(M, h)) \cong \fc_{L_{2d}}(\Gamma_{2d})$.
\section{Finite geometry and the group $\Gamma$}
From now on, we assume that $n=2$ in \eqref{bblattice}.
Where no confusion is likely to arise, we use $L$ to denote $L_{2d}$ and $\Gamma$ to denote $\Gamma_{2d}$.
\begin{prop}\label{modgprop}
  If $h \in M$ corresponds to a polarisation of degree $2d>4$ then
  \beq
  \Gamma_{2d} \cong \opn{O}^+(M, h).
  \eeq
Furthermore, if $d=p^2$ for prime $p>3$, then
  $\Gamma_{2d} \subset \Gamma_2$.
\end{prop}
\begin{proof}
  (c.f. Proposition 3.12 of \cite{handbook}.)
  As $\opn{O}(M, h)$ acts on both $\la h \ra$ and $\la h \ra^{\perp}$ but trivially on $\la h \ra$, we can immediately identify $\opn{O}(M, h)$ with a subgroup of $\opn{O}(L_{2d})$.

  The series of overlattices
  \beq
  \la h \ra \op h^{\perp} \subset M \subset M^{\vee} \subset \la h \ra^{\vee} \op (h^{\perp})^{\vee}.
  \eeq
  defines a series of inclusions of abelian groups
  \beq
  M/( \la h \ra \op \la h \ra^{\perp} )
  \subset
  \la h \ra^{\vee}/ \la h \ra \op (\la h \ra^{\perp})^{\vee} / ( \la h \ra^{\perp} )
  =
  D(\la h \ra) \op D(\la h \ra^{\perp} ).
  \eeq
  We can therefore regard the isotropic subgroup $H = M / ( \la h \ra \op h^{\perp} )$ as a subgroup of $D(\la h \ra) \op D(\la h \ra^{\perp})$, and define  corresponding projections
  $p_h:H \rightarrow D(\la h \ra)$ and
  $p_{h^{\perp}}:H \rightarrow D(\la h \ra^{\perp})$.
  Without loss of generality (as $h$ is split)  we can assume that $h = e_3 + df_3 \in U \op \la -6 \ra$ where $\{e_i, f_i \}$ is the canonical basis for the $i$-th copy of $U$ in $M$.
  Let
  $k_1 = e_3 - df_3$,
  $k_1' = (2d)^{-1} k_1$,
  $k_2' = (6)^{-1}k_2$ and
  $k_3' = (2d)^{-1} h$,
  where $k_2$ generates the $\la -6 \ra$ factor of $M$.
  Take a basis $\{e_1, f_1, e_2, f_2, k_1', k_2' \}$ for $(h^{\perp})^{\vee}$.
  By direct calculation,
  $H = \la k_3' - k_1', d(k_1' + k_3') \ra + (\la h \ra \op h^{\perp} )$,
  $p_{h^{\perp}}(H) = \la k_1' \ra$ and
  $D(h^{\perp}) = \la k_1' \ra \op \la k_2 ' \ra$.
  By Corollary 1.5.2 of \cite{Nikulin},
  \beq
  \opn{O}^+(M, h) \cong \{ g \in \opn{O}^+(h^{\perp}) \mid g \vert_{p_{h^{\perp}}(H)} = \opn{id} \}
  \cong \Gamma_{2d}, 
  \eeq
  and the first part of the claim follows.

  For the second part of the claim, let $p$ be an odd prime and embed $L_{2p^2} \subset L_2$ by identifying factors of $2U \op \la -6 \ra$ and mapping
  \beq
  L_{2p^2} \ni t + ak_1 \mapsto t + apk \in L_2
  \eeq
  where $t \in 2U \op \la -6 \ra$, $k$ generates $\la -2 \ra \subset L_2$ and $a \in \ZZ$.
  Define the totally isotropic subgroup $N \subset D(L_{2p^2})$ by $N = L_2 / L_{2p^2} \subset D(L_{2p^2})$.
  If $g \in \Gamma_{2p^2}$ then $g(k_1') = k_1' + L_{2p^2}$.
  As $N \subset \la k_1' \ra + L_{2p^2} \subset D(L_{2p^2})$ and $g(L_{2p^2}) = L_{2p^2}$ then $g$ preserves $N$ and so extends to a unique element of $\opn{O}(L_2)$.
  To verify $g \in \Gamma_2$ one notes that
  the dual of the $\la -2 \ra$ factor in $L_2$ is generated by $pk_1'$ and
  \begin{align*}
    g(pk_1')
    & \equiv p k_1 ' \bmod{L_{2p^2}} \\
    & \equiv p k_1' \bmod{L_2},
  \end{align*}
  from which the result follows.
\end{proof}
\begin{lem}\label{pmvlem}
  Suppose $p>3$ is prime and let $L=L_{2p^2}$.
  \begin{enumerate}
  \item If $g \in \opn{O}(L)$ then $g \underline{v}^* \equiv \pm \underline{v}^* \bmod{ L}$ and $g \underline{w}^* \equiv \pm \underline{w}^* \bmod{ L}$;
  \item $\Gamma_{2p^2} = \widetilde{\opn{O}}^+(L) \rtimes \la \sigma_{\underline{w}} \ra$.
  \end{enumerate}
\end{lem}
\begin{proof}
  We begin by calculating the elements of length $-1/2p^2 \bmod{ 2 \ZZ}$ in $D(L)$.
  The group $D(L) \cong C_6 \op C_{2p^2}$ and
  \begin{equation}\label{qldef}
  q_{L}(a,b) = -\frac{a^2}{6} - \frac{b^2}{2p^2} \bmod{ 2\ZZ}
  \end{equation}
  for  $(a,b) \in D(L)$.
  Suppose $(a,b) \in D(L)$ is of order $2p^2$ and length $-1/2p^2 \bmod{ 2\ZZ}$.
  As the order of $(a,b)$ is coprime to 3 then $a=0$ or $3$.
  If $a=0$ then
  \begin{equation}\label{blen}
  \frac{b^2}{2p^2} \equiv \frac{1}{2p^2} \bmod{2 \ZZ}
  \end{equation}
  or, equivalently, $(b+1)(b-1) \equiv 0 \bmod{ 4p^2}$.
  For order reasons, $(b, 2p)=1$ and so
  precisely one of $b \pm 1 \equiv 0 \bmod{ p}$ is true.
  If $b \equiv \pm 1 + xp \bmod{ p^2}$ for $x\in \ZZ$ then, from \eqref{blen}, $x \equiv 0 \bmod{ p}$.
  Similarly, as $2 b \not \equiv 0 \bmod{ 4}$ then $b$ is odd.
  Therefore, by the Chinese remainder theorem, $(0,b) = (0, \pm 1)$.
  The case $a=3$ cannot occur.
  From \eqref{qldef}, $3p^2 + b^2 \equiv 1 \bmod{ 4}$ and, as $p$ is odd,  we obtain the contradiction $b^2 \equiv 2 \bmod{ 4}$.
  We conclude that $D(L)$ contains two elements of order $2p^2$ and length $-1/2p^2 \bmod{ 2\ZZ}$, given by $\pm \underline{v}^*$.
  If $g \underline{w}^* =: (a,b) \in D(L)$ then $(g \underline{w}^*, g \underline{v}^*) \equiv \pm (g \underline{w}^*, \underline{v}^*) \equiv 0 \bmod{ \ZZ}$.
  As $((a,b), (0,1)) \equiv b/2p^2 \bmod{ \ZZ}$ then $b \equiv 0 \bmod{ 2p^2}$ and $a \equiv \pm 1 \bmod{ 6}$, from which the first claim follows.
  The second claim is immediate from Proposition  \ref{modgprop}.
\end{proof}
We now use an idea in \cite{Kondo} (attributed to O'Grady) to bound the index $\vert \Gamma_2 : \Gamma_{2p^2}  \vert$.
The approach involves considering the  quadratic space
\beq
\qc_p  := L_{2} / pL_2,
\eeq
 over the finite field $\FF_p$, where the quadratic form of $\qc_p$ is obtained by reducing the quadratic form of $L_2$ modulo $p$.
We shall require a number of classical results on orthogonal groups of finite type, which we state below for the convenience of the reader \cite{Dieudonne}.

Let $V_{\theta}$ denote the quadratic space $\la u, v \ra$ whose bilinear form is given by $(u,u)=1$, $(u, v)=0$ and $(v, v) = \theta$ for $-\theta \notin (\FF_p^*)^2$.
A non-degenerate quadratic space $V$ over $\FF_p$  is uniquely determined by  $\opn{dim} V$ and $\Delta := \opn{det }B \in \FF_p^* / (\FF_p^*)^2$ where $B$ is the bilinear form on $V$.
If $\opn{dim} V=2m$ and $\epsilon= (-1)^m \Delta \in \FF_p^*/(\FF_p^*)^2$ then $V$ is isomorphic to
\beq
\begin{cases}
V_{\epsilon}^{2m}  = H_1 \oplus \ldots \oplus  H_m & \text{if $\epsilon =1$} \\
 V_{\epsilon} ^{2m} = V_{\theta} \oplus H_1 \oplus H_2  \oplus \ldots \oplus H_{m-1} & \text{if $\epsilon = -1$}
\end{cases}
\eeq
where $H_i$ denote hyperbolic planes;
if $\opn{dim} V = 2m+1$ then there is a single isomorphism class for $V$, given by
\beq
V^{2m+1}=H_1 \op \ldots \op H_m \op \la \theta \ra
\eeq
for $0 \neq \theta \in \FF_p^*/(\FF_p^*)^2$.
The orthogonal groups $\opn{O}(V^{2m+1})$ and $\opn{O}(V_{\epsilon}^{2m})$ are of order
\begin{equation}\label{orthgporder}
\begin{cases}
\vert \opn{O}(V^{2m+1})\vert = 2p^{m^2} \prod_{i=1}^m (p^{2i} - 1) \\
\vert \opn{O}(V_{\epsilon}^{2m}) \vert = 2 p^{m(m-1)} (p^m - \epsilon) \prod_{i=1}^{m-1} (p^{2i} - 1).
\end{cases}
\end{equation}
\begin{lem}\label{translemma}
  For prime $p>3$ suppose non-isotropic $u, v  \in \qc_p$  define  hyperplanes $\Pi_u, \Pi_v \subset \qc_p$ given by $\Pi_u \perp u$ and $\Pi_v \perp v$.
  If $u^2/v^2 \in (\FF_p^*)^2$ then $\Pi_u$ and $\Pi_v$ are equivalent under $\opn{O}(L_2)$.
\end{lem}
\begin{proof}
Let $\{e_1, f_1, e_2, f_2, v_1, v_2 \}$ be a $\ZZ$-basis for $L_2$ where $v_1$, $v_2$ are generators for $\la -6 \ra$ and $\la -2 \ra$, respectively and $\{e_i, f_i \}$ are canonical bases for the two copies of $U \subset L_2$.
We begin by defining some elements of $\opn{O}(L_2)$.
For isotropic $e \in L_2$  and any $a \in e^{\perp} \subset L_2$, there exists   $t(e,a) \in \opn{O}(L_2)$ (an \emph{Eichler transvection}), defined by
\begin{equation}\label{eichlertrans}
t(e, a) : w \mapsto w - (a,w)e + (e,w)a - \frac{1}{2}(a,a)(e,w)e
\end{equation}
for $w \in L_2$ \cite{eichler, abelianisation}.
As $\opn{O}(2U) \subset \widetilde{\opn{O}}(L_2)$, we can also extend elements of  $\opn{O}(2U)$.
As is well known (e.g. \cite{Sterk}), if $(w,x,y,z) \in 2U$ (with respect to  canonical bases of $U$) then the map
\begin{equation}\label{sliso}
U \op U \ni (w,x,y,z) \mapsto
\begin{pmatrix}
 w & -y \\
 z & x
\end{pmatrix} \in M_2(\ZZ) 
\end{equation}
identifies $2U$ with $M_2(\ZZ)$, where the quadratic form on $M_2(\ZZ)$ is given by $2 \opn{det}$.
Therefore, any $(A, B) \in \opn{SL}(2, \ZZ) \times \opn{SL}(2, \ZZ)$ defines an element of $\opn{O}(U \op U)$ by
\begin{equation}\label{O(2U)}
(A, B) :
\begin{pmatrix}
w & -y \\
z & x
\end{pmatrix}
\mapsto
A
\begin{pmatrix}
w & -y \\
z & x
\end{pmatrix}
B^{-1}.
\end{equation}

We now use \eqref{eichlertrans} and \eqref{O(2U)} to show that any $w = (w_1, w_2, w_3, w_4, w_5, w_6) \in L_2 / p L_2$  defining a non-degenerate hyperplane $\Pi_w \perp w$ can be put in a standard form.
The transvections $t(e_2, v_1)$ and $t(e_2, v_2)$ act on $w = (w_1, w_2, w_3, w_4, w_5, w_6) \in L_2$ by
\begin{equation*}
\begin{cases}
t(e_2, v_1): w \mapsto (w_1, w_2, w_3 + 3w_4 + 6w_5, w_4, w_5 + w_4, w_6  ) \\
t(e_2, v_2): w \mapsto (w_1, w_2, w_3 + w_4 + 2w_6, w_4, w_5, w_6 + w_4)
\end{cases}
\end{equation*}
and so, without loss of generality, we can assume $w_4 \neq 0$ by applying $t(e_2,v_1)$ or $t(e_2, v_2)$, or by permuting $\{w_1, w_2, w_3, w_4 \}$ using  elements of $\opn{O}(2U)$.
By rescaling $w$ so that $w_4 = 1$, and by
repeated application of $t(e_2, v_1)$ and $t(e_2, v_2)$,
$w$ can be transformed to an element of the form
$(w_1',w_2',w_3',w_4', 0, 0)$.
By the existence of the Smith normal form for \eqref{sliso}
\cite{Newman}, $w$ can be mapped to an element $(w_1'', w_2'', 0,0,0,0)$ using \eqref{O(2U)}.
By rescaling as necessary, we can assume $w$ is given by $(1, a, 0,0,0,0)$.
We next construct a map between hyperplanes.
Without loss of generality, assume that $u=(1, a, 0,0,0,0)$ and $v=(1, b, 0,0,0,0)$. By assumption, $ab^{-1} \in (\FF_p^*)^2$ and so there exists $\mu, \lambda \in \FF_p$ such that $(\mu u)^2 = (\lambda v)^2$.
We define $\hat{u}$ and $\hat{v}$ by $\hat{u}:=\mu u = (u_1, u_2, 0,0,0,0)$ and   $\hat{v}:=\lambda v = (v_1, v_2, 0,0,0,0)$.
Without loss of generality, assume that $\hat{u} - \hat{v} = (r,s,0,0,0,0)$ is non-zero and, by taking  representatives for $r,s$ modulo $p$, let
\begin{equation*}
q :=
\begin{cases}
r & \text{if $s=0$}\\
s & \text{if $r=0$} \\
\opn{gcd}(r,s) & \text{otherwise}.
\end{cases}
\end{equation*}
If $r_1, r_2, s_1, s_2 \in \ZZ$ are solutions to
$r_2 u_1 + r_1 u_2 \equiv q \bmod{ p}$ and
$s_2 v_1 + y_2 v_2 \equiv q \bmod{ p}$,
define
$u', v', w \in e_2^{\perp} \cap f_2^{\perp} \subset L_2$
by $u'=(r_1, r_2, 0,0,0,0)$,
$v'=(s_1, s_2, 0,0,0,0)$ and
$w=(q^{-1}r, q^{-1}s, 0,0,0,0)$.
Then, over $\FF_p$,
$(\hat{u}, u') = q$, $(\hat{v}, v') = q$ and
$t(e_2, v')t(f_2, w)t(e_2, u'): \hat{u} \mapsto \hat{v}$ (c.f. Proposition 3.3 of \cite{abelianisation}) and the result follows.
\end{proof}
\begin{prop}\label{finindexprop}
  If $p>3$ is prime then $\vert  \Gamma_2 : \Gamma_{2p^2} \vert \leq 2(p^5 + p^2)$.
  Therefore, there exists a finite (branched) covering
  \beq
  \fc_{L_{2p^2}}(\Gamma_{2p^2}) \rightarrow \fc_{L_2}(\Gamma_2).
  \eeq
\end{prop}
\begin{proof}
  (c.f. \S3 \cite{Kondo}.)
  By definition, if $v \in U \subset L_2$ is of length $v^2 = 2$ then   $\opn{O}(L_2) = \opn{O}^+(L_2) \rtimes \la \sigma_ v \ra$.
  The non-degenerate hyperplane $\Pi:=L_{2p^2} / L_2 \subset \qc_p$ is stabilised by $\sigma_v$ and so
\beq
  \vert \opn{O}(L_2) : \opn{Stab}_{\opn{O}(L_2)}(\Pi) \vert
  =
  \vert \opn{O}^+(L_2) : \opn{Stab}_{\opn{O}^+(L_2)}(\Pi) \vert.
  \eeq
  By Lemma \ref{translemma}, hyperplanes in $\qc_p$ have the same orbits under $\opn{O}(L_2)$ and $\opn{O}(\qc_p)$.
  Therefore
  \begin{align}\label{OL2index}
    \vert \opn{O}^+(L_2) : \opn{Stab}_{\opn{O}^+(L_2)}(\Pi) \vert
    & =
    \vert \opn{O}(\qc_p) : \opn{Stab}_{\opn{O}(\qc_p)}(\Pi) \vert \nonumber \\
    & = \vert \opn{O}(\qc_p) : \opn{O}(\Pi) \times C_2 \vert,
  \end{align}
  where the last line follows from Witt's theorem.
  As any element of $\opn{Stab}_{\opn{O}^+(L_2)}(\Pi)$ extends to $\opn{O}^+(L_{2p^2})$ then
  \begin{equation}\label{subchain}
  \widetilde{\opn{O}}^+(L_{2p^2}) \subset
  \Gamma_{2p^2} \subset
  \opn{Stab}_{\opn{O}^+(L_2)}(\Pi) \subset
  \opn{O}^+(L_{2p^2})
  \end{equation}
  and so
  \beq
  \vert \opn{O}^+(L_2) : \Gamma_{2p^2} \vert
  =
  \vert \opn{O}^+(L_2) : \opn{Stab}_{\opn{O}^+(L_2)}(\Pi) \vert
  \vert \opn{Stab}_{\opn{O}^+(L_2)}(\Pi) : \Gamma_{2p^2} \vert.
  \eeq
  By \eqref{subchain} and Lemma \ref{pmvlem},
  \beq
  \vert \opn{Stab}_{\opn{O}^+(L_2)}(\Pi) : \Gamma_{2p^2} \vert \leq
  \vert \opn{O}^+(L_{2p^2}): \widetilde{\opn{O}}^+(L_{2p^2}) \vert =
  4.
  \eeq
  By Proposition \ref{modgprop}, $\Gamma_{2p^2} \subset \Gamma_2$ and by Lemma \ref{gamma2gplem}, $ \opn{O}^+(L_2) = \Gamma_2 $.
  Therefore,
  \begin{align*}
    \vert \Gamma_2 : \Gamma_{2p^2} \vert
    &
    = \vert \opn{O}^+(L_2): \Gamma_{2p^2} \vert  \\
    &\leq  \vert \opn{O}^+(L_2) : \opn{Stab}_{\opn{O}^+(L_2)}(\Pi) \vert
    \vert \opn{Stab}_{\opn{O}^+(L_2)}(\Pi) : \Gamma_{2p^2} \vert  \\
    & \leq  4 \vert \opn{O}^+(L_2) : \opn{Stab}_{\opn{O}^+(L_2)}(\Pi) \vert 
    \intertext{then by \eqref{OL2index},}
    & \leq \frac{4 \vert \opn{O}(\qc_p) \vert }{ \vert \opn{O}(\Pi) \times C_2 \vert} \\
\intertext{and by \eqref{orthgporder},}
  & \leq
    \frac{8p^6(p^3 + 1)(p^4 - 1)(p^2-1)}
  {4p^4(p^4-1)(p^2-1)} \\
  & \leq 2(p^5 + p^2),
  \end{align*}
  and the result follows.
\end{proof}
\section{The Baily-Borel compactification of $\fc_{L_{2p^2}}(\Gamma_{2p^2})$}
In this section, we study the boundary components of $\fc_{L_{2p^2}}(\Gamma_{2p^2})^*$.
We begin by counting  boundary points in Lemma \ref{2p2lineslem} before defining invariants for boundary curves in  Proposition \ref{normformprop}.
We use these invariants to classify  boundary curves up to isomorphism in Theorem \ref{curvethm} and provide bounds for their number  in Corollary  \ref{curvecountcor}.
We finish by describing incidence relations in Theorem \ref{L2boundarythm} and \ref{2p2curvethm}.
Unless otherwise stated, $L:=L_{2p^2}$  and we assume $b_L = ( (-1/6) \op (-1/2p^2), C_6 \op C_{2p^2})$ for prime $p>3$.
\subsection{Boundary points}
We say that an element $x \in D(L)$ is \emph{isotropic} if $x^2 \equiv 0 \bmod{ 2\ZZ}$.
\begin{lem}\label{isoveclem}
  If $D(L) \cong C_6 \op C_{2p^2}$ then the isotropic elements of $D(L)$ are given by
  \beq
  \{ (0, 2kp), (3, (2k+1)p) \mid k \in \ZZ\} \subset D(L).
  \eeq
\end{lem}
\begin{proof}
  An element $(x, y) \in D(L)$ is isotropic if and only if
    \begin{equation}\label{isoeq}
    p^2 x^2 + 3y^2 \equiv 0 \bmod{ 12p^2}.
    \end{equation}
    As $(3,p)=1$ then $p \vert y$ and we define $y_1$ by $y = py_1$.
    As $p \equiv \pm 1 \bmod{ 6}$ then  $x^2 + 3y_1^2 \equiv 0 \bmod{ 6}$.
    By considering squares modulo 6, $x \equiv y \bmod{ 2}$ and either $x \equiv 0$ or  $3 \bmod{6}$.
    Therefore, as all elements of
    \beq
    \{ (0, 2kp), (3, (2k+1)p) \mid k \in \ZZ\} \subset D(L)
    \eeq
    satisfy \eqref{isoeq}, the result follows.
\end{proof}
\begin{lem}\label{2p2lineslem}
  Let $v \in L$ denote a primitive isotropic vector.
  Then there are 4 families of points in the boundary of $\fc_L(\Gamma)^*$, given by
  \begin{enumerate}
  \item  $p_1$ corresponding to  $v^* \equiv (0,0) \bmod{ L}$;
  \item $p_2$ corresponding to  $v^* \equiv (3,p^2) \bmod{ L}$;
  \item $p_p(k)$ corresponding to  $v^* \equiv (0, 2kp) \bmod{ L}$ for $k=0, \ldots, p-1$;
  \item $p_{2p}(k)$ corresponding to $v^* \equiv (3,(2k+1)p) \bmod{ L}$ for $k=0, \ldots, (p-3)/2$.
    \end{enumerate}
\end{lem}
\begin{proof}
  By Theorem \ref{bbdec}, points in the boundary of $\fc_L(\Gamma)^*$ are in bijection with $\Gamma$-orbits of primitive totally isotropic rank 1 sublattices of $L$.
  By Lemma \ref{isoveclem}, if $\pm v \in L$ is primitive and isotropic then $\pm v^* \in D(L)$ is given by 
  $(0,0)$ if $v^*$ is of order 1;
  $(3, p^2)$  if $v^*$ is of order 2;
  $(0, 2kp)$ for some $k=0, \ldots, p-1$  if $v^*$ is of order $p$; or
  $(3, (2k+1)p)$ for some $k=0, \ldots, (p-3)/2$  if $v^*$ is of order $2p$.
  By Proposition \ref{modgprop}, $\widetilde{\opn{SO}}^+(L) \subset \Gamma$ and so, by Lemma \ref{pmvlem} and the Eichler criterion, the $\Gamma$-orbits of primitive $\pm v \in L$ are uniquely determined by $ \pm v^* \bmod{ L}$ as above.

  We show that each case can occur.
  Take a basis $\{v_i \}_{i=1}^6$ where $\{v_1, v_2\}$, $\{v_3, v_4 \}$ are canonical bases for $U$ and $v_5:=\underline{w}$, $v_6:=\underline{v}$.
  \begin{enumerate}
\item By definition of $U$,  $v=(1,0,0,0,0,0) \in L$ is primitive, isotropic and $v^* \equiv (0,0) \bmod{ L}$.
\item If $v \in L$ is of the form $v=(2, 0, 2v_3, 2v_4, 1, 1)$  then $\opn{div}(v) = 2$ and $v$ is primitive with  $v^* \equiv (3, p^2) \bmod{ L}$.
  As $p$ is prime then $p^2 \equiv 1 \bmod{ 8}$ and so
  $v^2=8v_3 v_4 + 6 + 2p^2 = 0$ admits an integral solution in $v_3$, $v_4$.
\item If $v \in L$ is of the form $v = (2p, 0, 2p v_3, 2p v_4, p, (2k+1)) \in L$ where $(2k+1, p) = 1$ then, as $(2k+1, p)=1$ and $(2,p)=1$, $v$ is primitive and $\opn{div}(v)=2p$.
  One checks that  $v^* \equiv (3, (2k+1)p) \bmod{ L}$.
  As $p$ and $2k+1$ are odd then $(2k+1)^2 p^2 \equiv 1 \bmod{ 8}$ and  $v^2=8p^2 v_3 v_4 + 2(2k+1)^2 p^2  = 0$ admits an integral solution in $v_3$, $v_4$ for each $k$.
\item  If $v \in L$ is of the form $v = (p,0, pv_3, pv_4, 0, k) \in L$ with $(k,p)=1$ then $v$ is primitive and  $v^* \equiv (0, 2kp) \bmod{ L}$.
  As $v^2=p^2 v_3 v_4 + 2k^2 p^2 = 0$ admits an integral solution in $v_3$, $v_4$ for each $k$, the result follows.
  \end{enumerate}
\end{proof}
\subsection{Invariants associated with boundary curves}
We now show that there exists a normal form for the Gram matrix of $L$ with respect to a primitive totally isotropic sublattice $E \subset L$ of rank 2.
Our approach essentially follows that of \cite{Scattone}.
\begin{prop}\label{normformprop}
  Let $E \subset L$ be a primitive totally isotropic sublattice of rank 2. Then there exists a $\ZZ$-basis $\{v_i \}_{i=1}^6$ of $L$ such that $\{v_1, v_2 \}$ is a $\ZZ$-basis for $E$ and $\{v_1, \ldots, v_4 \}$ is a $\ZZ$-basis for $E^{\perp}$.
  The basis can be chosen such that the Gram matrix
  \begin{equation}\label{eqnform}
    Q
    =
    ((v_i, v_j))
    =
    \bpm
    0 & 0 & A \\
    0 & B & C \\
    {^{\intercal}}A & {^{\intercal}}C & D
    \epm
  \end{equation}
    where    $B$ is the bilinear form on $E^{\perp}/E$,
    \begin{center}
        \begin{tabular}{c c c}
          $A = \bpm 0 & a \\ 1 & 0 \epm$
          & and &
          $D = \bpm d & 0 \\ 0 & 0 \epm$
        \end{tabular}
        \end{center}
    where $a = 1$, $2$, $p$ or $2p$ and $d \in 2 \ZZ$ is taken modulo $2a$.
    Furthermore,
    \begin{enumerate}
    \item if $a=1$ then $C=D=0$;
    \item if $a=2$ then $C$ can be taken modulo 2 and $d=0$ or $2$;
    \item if $a=p$ then $C=0$ and $B \cong \la -2 \ra \op \la -6 \ra$ or $B \cong -\bsm 4 & 2 \\ 2 & 4 \esm$;
    \item if $a=2p$ then $C=0$ and $B \cong A_2(-1)$.
    \end{enumerate}
  \end{prop}
  \begin{proof}
    As the lattices $E$ and $E^{\perp}$ are primitive,  there exists a $\ZZ$-basis of $L$ with Gram matrix of the form \eqref{eqnform}.
As $\vert \opn{det}(Q) \vert = 12p^2$ then $\opn{det}(A)$ is square-free and thus given  by $1$, $2$, $p$ or $2p$.
        By the existence of the Smith normal form \cite{Newman},  one can apply a base change $\opn{diag}(P, I, Q)$ for $P, Q \in \opn{GL}(2, \ZZ)$ so that $A$ is as in the statement of the lemma.
        All cases of $A$ are realised: for example, one can take $v_1$ to be a primitive isotropic vector in $U$ and $v_2$ to be one of the vectors
        $(1,0,0,0)$,
        $(2, 2p^2, p, 1)$,
        $(p,p,0,1)$ or
        $(2p, 2p, p, 1)$ in
        $U \op \la -6 \ra \op \la -2p^2 \ra$, of divisor $1$, $2$, $p$ and $2p$, respectively.
        We now refine the basis further.
        \begin{enumerate}
    \item Suppose $a=1$.
      From \eqref{eqnform}, $\opn{div}(v_1)=1$ and, from the classification of unimodular lattices, $v_1 \in U$.
      Similarly, $\opn{div}(v_2)=1$ and $v_2 \in U^{\perp} \subset L$.
      Therefore, by Proposition 1.15.1 of \cite{Nikulin}, there exists a sublattice $U \op U \op L' \subset L$ with $v_1$ and $v_2$ each contained in a copy of $U$.
      As $\vert \opn{det}(L') \vert = \vert \opn{det}(L) \vert$ then  $L = 2U \op L'$ and we conclude $C=D=0$.
    \item Suppose $a=2$.
      As above, we can assume that $v_1 \in U \subset L$.
      From \eqref{eqnform}, $\opn{div}(v_2)=2$ and, from tables in \cite{SPLAG}, $v_5$ can be chosen such that
      \begin{center}
        \begin{tabular}{c c c}
          $D = \bpm 0 & 2 \\ 2 & 0 \epm$
          & or &
          $\bpm 0 & 2 \\ 2 & 2 \epm$.
        \end{tabular}
      \end{center}
      To reduce $C$ modulo $2$ we apply the base change
      \beq
      \bpm
      I & S & 0 \\
      0 & I & 0 \\
      0 & 0 & I
      \epm:
      Q \mapsto
      \bpm
      0 & 0 & A \\
      0 & B & {}^{\intercal}AS + C \\
      * & * & D
      \epm
      \eeq
      for an appropriate choice of $S$.
      To put $D$ in the required form, we apply 
      \begin{equation}\label{Dtrans}
        \bpm
        I & 0 & W \\
        0 & I & 0 \\
        0 & 0 & I
        \epm:
        Q \mapsto
        \bpm
        0 & 0 & A \\
        0 & B & C \\
        * & * &  {^{\intercal}}WA +{^{\intercal}}AW  + D
        \epm,
      \end{equation}
      and the result follows by noting that $L$ is even and
      \beq
      \left \{ {}^{\intercal}WA +{^{\intercal}}AW \mid W \in M_2(\ZZ) \right \}
      =
      \left \{
      \bpm
      2ax & y \\
      y & 2z
      \epm
      \mid
      x,y,z \in \ZZ
      \right \}.
      \eeq
    \item Suppose $a=p$.
      From \eqref{eqnform}, $\vert \opn{det}(B) \vert =12$ and so, from tables in \cite{SPLAG},
      \begin{center}
        \begin{tabular}{c c c}
          $B \cong \la -2 \ra \op \la -6 \ra$
          & or &
          $-\bpm 4 & 2 \\ 2 & 4 \epm$.
        \end{tabular}
      \end{center}
      To put $C$ in the correct form, we apply the base change
      \begin{equation}\label{Ctrans}
        \bpm
        I & S & 0 \\
        0 & I & T \\
        0 & 0 & I
        \epm:
        Q \mapsto
        \bpm
        0 & 0 & A \\
        0 & B & {}^{\intercal}SA + BT + C \\
        * & * & D
        \epm.
      \end{equation}
      If $C = (c_{ij})$,
      $S = (s_{ij})$,
      $T = (t_{ij})$ and 
      $B = \la -2 \ra \op \la - 6 \ra$,
      then
      \beq
          {}^{\intercal}SA + BT + C =
          \bpm
          as_{21} - 2 t_{11} + c_{11} &
          s_{11} - 2t_{12} + c_{12} \\
          a s_{22} - 6 t_{21} + c_{21} &
          s_{12} - 6t_{22} + c_{22}
          \epm.
          \eeq
          As $(a, 6)=1$, there exists $T$ such that
          $-2t_{11} + c_{11} \equiv 0 \bmod{ a}$ and
          $-6t_{21} + c_{21} \equiv 0 \bmod{ a}$.
          Therefore, there exists $S$ such that $ {}^{\intercal}SA + BT + C =0$.

          Similarly, if $B = - \bpm 4 & 2 \\ 2 & 4 \epm$ then
          \beq
          {}^{\intercal}SA + BT + C =
          \bpm
          as_{21} - 4 t_{11} - 2t_{21} + c_{11} &
          s_{11} - 4t_{12} -2t_{22} + c_{12} \\
          a s_{22} - 2 t_{11} - 4t_{21} + c_{21} &
          s_{12} - 2t_{12} -4t_{22} + c_{22}
          \epm
          \eeq
          and the same conclusion follows.
          In either case, we put $D$ in the required form by applying an appropriate base change  \eqref{Dtrans}.
        \item If $a=2p$ then $\vert \opn{det}(B) \vert = 3$ and, from tables in \cite{SPLAG}, $B \cong A_2(-1)$.
          One then proceeds as for $a=p$.
    \end{enumerate}
  \end{proof}
  \begin{defn}
    If $E \subset L$ is a primitive totally isotropic sublattice of rank 2 and $a$ is as in Proposition \ref{normformprop}, we say that $E$ and the associated boundary curve $\cc_E$ are of type $a$.
  \end{defn}
\subsection{Geometry of boundary curves}
  We now study the groups $G(E) = \opn{Stab}_{\Gamma}(E) / \opn{Fix}_{\Gamma}(E)$ in order to classify the curves $\cc_E$ up to  isomorphism.
  We assume throughout that $E \subset L$ is a primitive totally isotropic sublattice of rank 2 and type $a$.
  \begin{defn}
    If $g \in \opn{Stab}_{\opn{O}(L)}(E)$ then, on the basis of Proposition \ref{normformprop},
    \begin{equation}\label{gdef}
    g =
    \bpm
    U & V & W \\
    0 & X & Y \\
    0 & 0 & Z
    \epm.
    \end{equation}
    We define the homomorphism
    $\pi_E: \opn{Stab}_{\opn{O}(L)}(E) \rightarrow \opn{GL}(2, \ZZ)$
    by $\pi_E:g \mapsto U$.
  \end{defn}
For $n \in \NN$, let $\Gamma(n) \subset \opn{SL}(2, \ZZ)$ denote the principal congruence subgroup of level $n$ and let  
\beq
\Gamma_0(n) =
\left \{
Z \in \opn{SL}(2, \ZZ) \mid
Z \equiv \bpm * & 0 \\ * & * \epm \bmod{ n}
\right \}
\eeq
and
\beq
\Gamma_1(n) =
\left \{
Z \in \opn{SL}(2, \ZZ) \mid
Z \equiv \bpm 1 & 0 \\ * & 1 \epm \bmod{ n}
\right \}.
\eeq
\begin{lem}\label{piEsublem}
  If $g \in \opn{Stab}_{\Gamma}(E)$ then $\pi_E(g) \in \opn{SL}(2, \ZZ)$ if $a=1$ and $\pi_E(g) \in \Gamma_1(a)$ otherwise.
  \end{lem}
  \begin{proof}
    Suppose $g \in \opn{Stab}_{\Gamma}(E)$ is as in \eqref{gdef}.
    Then, by Lemma 5.7.1 of \cite{Brieskorn}, $g \in \opn{O}^+(L)$ if and only if $U \in \opn{SL}(2, \ZZ)$.
    As ${^{\intercal}}gQg=Q$ then ${}^{\intercal}UAZ = A$ and so, if
    \begin{center}
      \begin{tabular}{c c c}
        $Z = \bpm r & s \\ t & u \epm$
        & then &
        $U = \bpm r & -as \\ -a^{-1}t & u \epm$.
      \end{tabular}
    \end{center}
    Therefore, $U \in \Gamma_0(a)$ if $a \neq 1$ and $U \in \opn{SL}(2, \ZZ)$ otherwise.

    Let $\{ v_i \}_{i=1}^6$ be the basis defined in Proposition \ref{normformprop}.
    By Lemma \ref{pmvlem}, $g$ acts trivially on $C_2 \op C_2 \subset D(L)$.
    If $a=2$ then, as $\opn{div}(v_2)=2$, $gv_2^* \equiv v_2^* \bmod{ L}$, implying $U \in \Gamma_1(2)$.
    By definition of $\Gamma$, if $a=p$ or $2p$ then $g$ acts trivially on $C_p \subset D(L)$.
    Therefore, by considering the action of $g$ on $v_2^*$, we conclude $U \in \Gamma_1(a)$.
  \end{proof}
\begin{lem}\label{piEimglem}
    The image
    \beq
    \pi_E(\opn{Stab}_{\Gamma}(E))
    =
    \begin{cases}
      \opn{SL}(2, \ZZ) & \text{if $a=1$} \\
      \Gamma_1(a) & \text{if $a=2, p$ or $2p$}.
    \end{cases}
    \eeq
    \end{lem}
  \begin{proof}
    We construct a pre-image for $\pi_E$.
    Let
    \beq
    Q'=
    \bpm 0 & A \\ {}^{\intercal}A & D \epm
    \eeq
    be the Gram matrix of $L':=\la v_1, v_2, v_5, v_6 \ra \subset L$ where $\{ v_i \}_{i=1}^6$ is the $\ZZ$-basis of $L$ defined in Proposition \ref{normformprop}.
    Suppose $U \in \opn{SL}(2, \ZZ)$ if $a=1$ and $U \in \Gamma_1(a)$ otherwise.
    Assume that $Z \in \opn{SL}(2, \ZZ)$ satisfies ${}^{\intercal}UAZ = A$.
    Proceeding along the lines of  \cite{Scattone},  we show that there exist elements of the form
  \beq
  g = \bpm U & UW \\ 0 & Z \epm \in \opn{O}^+(L')
  \eeq
  extending to $\opn{Stab}_{\Gamma}(E)$.
  As
  \beq
      {^{\intercal}}g Q' g =
      \bpm 0 & {^{\intercal}}UAZ \\
           {^{\intercal}}Z^{\intercal}AU & {^{\intercal}}WA +{^{\intercal}}AW + {^{\intercal}}ZDZ
           \epm
           \eeq
           and ${}^{\intercal}UAZ = A$, then $W$ must satisfy
           \begin{equation}\label{WDeq}
             {^{\intercal}}WA +{^{\intercal}}AW + {^{\intercal}}ZDZ = D.
             \end{equation}
If $W = (w_{ij})$ and
\beq
Z = \bpm r & s \\ t & u \epm
\eeq
then
 \begin{equation}\label{Weq}
{}^{\intercal}WA +{^{\intercal}}AW  +{^{\intercal}}ZDZ =
    \bpm
    2aw_{21} +dr^2 & aw_{22} + w_{11} + drs \\
    aw_{22} + w_{11} + drs  & 2w_{12} + ds^2
    \epm.
 \end{equation}
 Equation \eqref{WDeq} is always satisfied for some $W$: \begin{enumerate}
 \item if $a=1$ or $2$ and $D=0$, set $W:=0$;
 \item otherwise, $d$ is even (as $L$ is even) and, by  Lemma \ref{piEsublem},  $r^2 \equiv 1 \bmod{ a}$.
\end{enumerate}
 By \eqref{Weq}, there exists $W$ satisfying \eqref{WDeq} in both cases.
We now show  that $g \in \opn{O}(L')$ can be extended to $\Gamma$ by allowing $g$ to act trivially on $(L')^{\perp} \subset L$.
 We note that as $U \in \opn{SL}(2, \ZZ)$ then, by Lemma 5.7.1 of \cite{Brieskorn}, the extension of $g$ automatically belongs to $\opn{O}^+(L \otimes \RR)$.
 \begin{enumerate}
 \item If $a=1$ or $2$ and $D=0$ then $g \in \widetilde{\opn{SO}}^+(L') \subset \widetilde{\opn{O}}^+(L) \subset \Gamma$.
 \item If $a=2$ and $d=2$ then $\opn{O}(D(L'))$ is trivial and so $g \in \widetilde{\opn{O}}^+(L') \subset \widetilde{\opn{O}}^+(L) \subset \Gamma$.
 \item If $a=p$ or $2p$ then, by Proposition \ref{normformprop}, there exists a splitting $L = L' \op B$.
   By construction, $g$ acts trivially on the element $v_2^* \in D(L)$ generating the subgroup $C_p \subset D(L)$.
   Therefore, by Lemma \ref{pmvlem}, $g$ acts trivially on $C_{p^2} \subset D(L)$ and fixes the subgroup $C_2 \op C_2 \subset D(L)$.
   Therefore, by Proposition \ref{modgprop}, $g \in \Gamma$.
 \end{enumerate}
\end{proof}
\begin{thm}\label{curvethm}
  If $\cc_E$ is the boundary curve of $\fc_L(\Gamma)^*$ corresponding to $E$, then
  \beq
  \cc_E \cong
  \begin{cases}
    \HH^+/\opn{PSL}(2, \ZZ) & \text{if $a=1$} \\
    \HH^+/\Gamma_1(a) & \text{otherwise.}
    \end{cases}
  \eeq
\end{thm}
\begin{proof}
  Immediate from Theorem \ref{bbdec} and Lemma \ref{piEimglem}, as  $\pi_E(\opn{Fix}_{\Gamma}(E)) \subset \la \pm I \ra$.
\end{proof}
\subsection{Counting boundary curves}
As a corollary to Proposition \ref{normformprop}, we can bound the number of boundary curves in $\fc_L(\Gamma)^*$.
We assume $L=L_{2p^2}$ and $\Gamma= \Gamma_{2p^2}$ for prime $p>3$.
\begin{cor}\label{curvecountcor}
   If $h(D)$ is the class number of discriminant $D$, then
   the boundary of $\fc_L(\Gamma)^*$ contains at most
 $4 h(-48 p^2)$ curves of type $1$,
 $128 h(-12 p^2)$ curves of type $2$,
 $8a$ curves of type $p$ and
 $4a$ curves of type $2p$.
\end{cor}
\begin{proof}
  By Theorem \ref{bbdec}, it suffices to bound the number of $\Gamma$-equivalence classes of primitive totally isotropic sublattices of rank 2 in $L$.
  In each case, we first count the number of Gram matrices occurring in  Proposition \ref{normformprop} for each $a$, to obtain bounds for equivalence in $\opn{O}(L)$.
  We note that there are at most $h(48p^2/a^2)$ choices for $B$ for a  given $a$.
  By Lemma \ref{pmvlem},
\beq
\vert \opn{O}(L) : \Gamma \vert = \vert \opn{O}(L) : \opn{O}^+(L) \vert \vert \opn{O}^+(L) : \Gamma \vert = 4,
\eeq
from which we obtain a bound for equivalence in $\Gamma$.
\end{proof}
\subsection{The boundary of $\fc_{L_2}(\Gamma_2)^*$}
To provide a specific example, we describe the boundary of  $\fc_{L_2}(\Gamma_2)^*$.
Let $L=L_2$ and $\Gamma=\Gamma_2$.
\begin{defn}(\cite{Brieskorn})
  If $E \subset L$ is a primitive totally isotropic sublattice, let $H_E:=E^{\perp \perp}/E \subset D(L)$ where $E^{\perp \perp} \subset L^{\vee}$.
\end{defn}
\begin{lem}\label{HElem}
  If $E \subset L$ is a primitive totally isotropic sublattice of rank 2, then
$E^{\perp}/E \cong \la -6 \ra \op \la -2 \ra$ or $E^{\perp}/E \cong A_2(-1)$.
\end{lem}
\begin{proof}
The lattice $E^{\perp}/E$ is negative definite and, by Lemma 4.1 of \cite{Brieskorn}, $D(E^{\perp}/E) \cong H_E^{\perp}/H_E$.
If $(a,b) \in D(L) \cong C_6 \op C_2$ is isotropic then $a^2 / 6 + b^2 / 2 = 0 \bmod{2 \ZZ}$ and so $(a,b)=(0,0)$ or $(3,1)$.
Therefore, $H_E = \la (0,0) \ra$ or $\la (3,1) \ra$.
If $H_E = \la (0,0) \ra$ then $D(E^{\perp}/E)$ has discriminant form $((-1/6) \op (1/2), C_6 \op C_2)$.
By tables in \cite{SPLAG}, the two negative definite even lattices of determinant 12 are
\begin{center}
  \begin{tabular}{c c c}
$\la -6 \ra \op \la -2 \ra$
& and &
$ \bpm
    -4 & -2 \\
    -2 & -4
    \epm$,
  \end{tabular}
  \end{center}
with only $\la -6 \ra \op \la -2 \ra$ having discriminant form  $( (-1/6) \op (-1/2), C_6 \op C_2)$.
Therefore,  $E^{\perp} / E \cong \la -6 \ra \op \la -2 \ra$.
If $H_E = \la (3,1) \ra$ then $H_E^{\perp} = \la (1,1) \ra$ and $D(E^{\perp}/E)$ has discriminant form $( (-2/3), C_3)$.
Therefore, from tables in \cite{SPLAG}, $E^{\perp}/E \cong A_2(-1)$.
\end{proof}
\begin{lem}\label{normform2lem}
  Assuming the notation of Proposition \ref{normformprop}, if $E \subset L$ is a primitive totally isotropic sublattice of rank 2 then there exists a $\ZZ$-basis $\{v_1, \ldots, v_6 \}$ of $L$ such that $\{v_1, v_2 \}$ is a $\ZZ$-basis for $E$ and
  $\{v_1, \ldots, v_4 \}$ is a $\ZZ$-basis for $E^{\perp} \subset L$ with Gram matrix as in \eqref{eqnform}.
  Furthermore, if $H_E$ is trivial then $a=1$, $B = \la -6 \ra \op \la -2 \ra$ and $C=D=0$;
  otherwise, $a=2$, $B = A_2(-1)$, $C=0$ and $d=2$.
\end{lem}
\begin{proof}
  As in Proposition \ref{normformprop}, there exists a basis with Gram matrix
  \begin{equation}\label{Q2def}
    Q =
    \bpm
    0 & 0 & A \\
    0 & B & C \\
    {}^{\intercal}A & {}^{\intercal}B & D
  \epm.
  \end{equation}
  By Lemma \ref{HElem}, $B \cong \la -6 \ra \op \la -2 \ra$ if  $H_E$ is trivial and $B \cong A_2(-1)$ otherwise.
  The case of trivial $H_E$ proceeds identically to the case of $a=1$ in Proposition \ref{normformprop}.
  If $H_E = C_2$ then, from \eqref{Q2def} and the existence of the Smith normal form, we can assume that
  \beq
  A  = \bpm 0 & 2 \\ 1 & 0 \epm.
  \eeq
  If $S, T \in M_2(\ZZ)$ then
  \beq
      {}^{\intercal}S A + BT + C_1 =
      \bpm
      2 s_{21} - 2 t_{11} - t_{21} + c_{11} & s_{11} - 2t_{12} - t_{22} + c_{12} \\
      2s_{22} - t_{11} - 2t_{21} + c_{21}   & s_{12} - t_{12} - 2t_{22} + c_{22}
      \epm,
   \eeq
   and so, by applying a base change of the form \eqref{Ctrans}  we can assume $C=0$.
   Similarly, as
\beq
  \left \{{^{\intercal}}WA+{^{\intercal}}AW \mid W \in M_2(\ZZ) \right \}
  =
  \left \{
  \bpm
  4a & b \\
  b & 2c
  \epm
  \mid a,b,c \in \ZZ \right \},
  \eeq
  there exists a base change \eqref{Dtrans} reducing $D$ to $\opn{diag}(d, 0)$ where $d = 0$ or $2$.
   As $2U \subset L$, then $L$ is unique in its genus and so  uniquely determined by its signature and discriminant form \cite{Nikulin}.
   Therefore, by comparing the discriminant forms defined by \eqref{Q2def} for $d=0$ and $d=2$, only the case $d=2$ occurs.
\end{proof}
\begin{lem}\label{gamma2gplem}
  The group $\Gamma_2 = \opn{O}^+(L)$.
\end{lem}
\begin{proof}
  We first calculate $\opn{O}(D(L))$.
  The group $D(L) \cong C_2 \op C_2 \op C_3$ and if $(a,b,c) \in D(L)$ then
  \begin{equation}\label{qL2}
  q_{L}(a,b,c) = -\frac{a^2}{2} - \frac{3b^2}{2} - \frac{2c^2}{3} \bmod{ 2\ZZ}.
  \end{equation}
  The three elements of order 2 in $D(L)$ are of length
  $q_{L}(1,0,0) \equiv -1/2 \bmod{ 2\ZZ}$,
  $q_{L}(0,1,0) \equiv -3/2 \bmod{ 2\ZZ}$ and
  $q_{L}(1,1,0) \equiv 0 \bmod{ 2\ZZ}$.
  Therefore, $\opn{O}(D(L))$ fixes the subgroup $C_2 \op C_2 \subset D(L)$ and acts as $\pm 1$ on $C_3$.
  Therefore, $\opn{O}^+(L):\widetilde{\opn{O}}^+(L) = \{ e, \sigma_{\underline{w}} \}$ where $\sigma_{\underline{w}}$ is the reflection defined by $\underline{w} \in L$ generating the $\la -6 \ra$ factor of $L$.
  As $\sigma_{\underline{w}} \in \Gamma_2$, the result follows by Proposition \ref{modgprop}.
\end{proof}
\begin{lem}\label{L2planeslem}
  There are two $\Gamma_2$-orbits of primitive totally isotropic sublattices of rank 2 in $L$.
\end{lem}
\begin{proof}
  By Lemma \ref{normform2lem}, there are two $\opn{O}(L_2)$-orbits of primitive totally isotropic sublattices of rank 2 in $L$, which are uniquely determined by the groups $H_E$.
  We take representatives $E_1$ and $E_2$ for each orbit, where
  $E_1 = \la e_1, e_2 \ra$, $E_2 = \la e_1, l \ra$ and  $l=2e_2 + 2f_2 + \underline{v} + \underline{w}$.
  If $x=e_1 + f_1$ and $y = e_1 - f_1$ then
  $\opn{sn}_{\RR}(\sigma_x) = -1$,
  $\opn{sn}_{\RR}(\sigma_y) = 1$ and
  one checks that
  $\sigma_y \sigma_x E_1 = E_1$
  and
  $\sigma_y \sigma_x E_2 = E_2$.
  As $\opn{O}(L_2): \opn{O}^+(L_2) = \{e,  \sigma_x\}$ and $\sigma_y \in \widetilde{\opn{O}}^+(L_2)$ then there are two $\widetilde{\opn{O}}^+(L)$-orbits of primitive totally isotropic sublattices of rank 2 in $L$, and the result follows by Lemma \ref{gamma2gplem}.
\end{proof}
\begin{lem}\label{L2lineslem}
  There are two $\Gamma_2$-orbits of primitive isotropic vectors in $L$.
\end{lem}
\begin{proof}
  By the Eichler criterion,  the $\widetilde{\opn{SO}}^+(L)$-orbits of primitive isotropic $v \in L$ are uniquely determined by $v^* \in D(L)$.
  If $v_i$ is isotropic in $L$ then $v_i^*$ is isotropic in $D(L)$.
  Let $D(L) \cong C_2 \op C_2 \op C_3$ with $q_{L}$ as in \eqref{qL2}.
  The only isotropic elements of $D(L)$ are $(0,0,0)$ and $(1,1,0)$.
  If $v_1 = e_1$ then $v_1^* = (0,0,0)$ and
  if $v_2 = 2e_2 + 2f_2 + \underline{v} + \underline{w} $ then $v_2^* = (1,1,0)$.
  By Proposition \ref{modgprop}, $\widetilde{\opn{SO}}^+(L) \subset \Gamma_2$ and, as $v_1^*$ and $v_2^*$ can never be equivalent under $\Gamma_2$, the result follows.
\end{proof}
\begin{figure}
  \centering
  \begin{tikzpicture}
\filldraw
    (0,0) circle (0.1)
    (-4,0) circle (0.1)
    (4,0) circle (0.1);
    \draw
    (-4,0)--(4,0)
    (0, 0) node[below]{$P_2$}
    (-2, 0) node[below]{$\cc_1$}
    (2, 0) node[below]{$\cc_2$}
    (4, 0) node[below]{$P_3$}
    (-4, 0) node[below]{$P_1$};
  \end{tikzpicture}
  \caption{The boundary of $\fc_{L_2}(\Gamma_2)^*$}\label{boundary}
\end{figure}
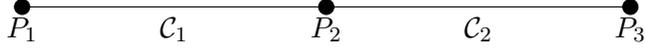
\begin{thm}\label{L2boundarythm}
The boundary of $\fc_{L_2}(\Gamma_2)^*$ consists of curves $\cc_1$ and $\cc_2$ of type $1$ and $2$, respectively and points $P_1$, $P_2$, $P_3$.
  As illustrated in Figure \ref{boundary}, the only intersections between boundary points and the closures of boundary curves are $\overline{\cc}_1 \cap P_1$, $\overline{\cc}_1 \cap P_2$, $\overline{\cc}_2 \cap P_2$ and $\overline{\cc}_2 \cap P_3$.
\end{thm}
\begin{proof}
  Immediate from Theorem \ref{bbdec}, Lemma \ref{L2planeslem} and Lemma \ref{L2lineslem}.
\end{proof}
\subsection{The boundary of $\fc_{L_{2p^2}}(\Gamma_{2p^2})^*$}
We now describe the boundary of $\fc_{L_{2p^2}}(\Gamma_{2p^2})^*$ in general.
We let $L=L_{2p^2}$ and $\Gamma=\Gamma_{2p^2}$ for prime $p>3$.
\begin{thm}\label{2p2curvethm}
The boundary of $\fc_L(\Gamma)^*$ consists of curves $\cc_a$ of type $1$, $2$, $p$ and $2p$, whose isomorphism classes are given by Theorem \ref{curvethm}; and boundary points $p_i$ and $p_i(k)$, as in Lemma \ref{2p2lineslem}.
    Furthermore, the closure of a boundary curve $\overline{\cc_a}$ contains $p_i$ or $p_i(k)$ if and only if $i \vert a$, as illustrated in Figure \ref{boundaryfp2}. \end{thm}
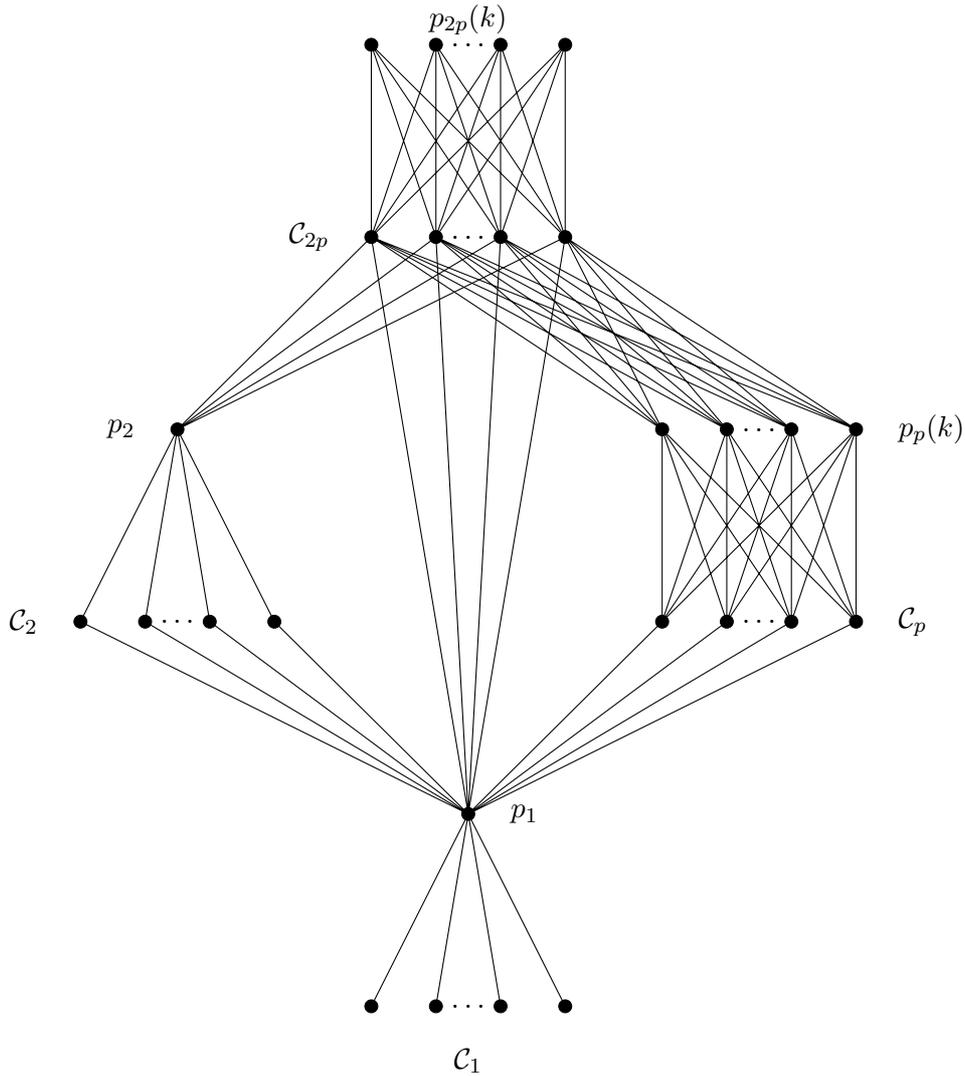
\begin{figure}[h!]
 \centering
  \begin{tikzpicture}[scale = 0.85]

\filldraw
(1.5,-1.5) circle (0.1)
(2.5,-1.5) circle (0.1)
(3.5,-1.5) circle (0.1)
  (4.5,-1.5) circle (0.1);
\draw
(1,-1.5) node [left] {$\cc_2$};

\filldraw
  (3,1.5) circle (0.1);
  \draw
(2.5,1.5) node [left] {$p_2$};

\draw
(3, -1.5) node {$\ldots$};
\draw
(1.5,-1.5) -- (3,1.5)
(2.5,-1.5) -- (3,1.5)
(3.5,-1.5) -- (3,1.5)
(4.5,-1.5) -- (3,1.5)
(3,1.5) -- (6,4.5)
(3,1.5) -- (7,4.5)
(3,1.5) -- (8,4.5)
(3,1.5) -- (9,4.5)
(1.5,-1.5) -- (7.5,-4.5)
(2.5,-1.5) -- (7.5,-4.5)
(3.5,-1.5) -- (7.5,-4.5)
(4.5,-1.5) -- (7.5,-4.5);

\filldraw
(6,-7.5) circle (0.1)
(7,-7.5) circle (0.1)
(8,-7.5) circle (0.1)
  (9,-7.5) circle (0.1);
  \draw
  (7.5, -7.5) node {$\ldots$};
  \draw
(7.5,-8) node [below] {$\cc_1$};

\draw
(7.5,-4.5) -- (6,-7.5)
(7.5,-4.5) -- (7,-7.5)
(7.5,-4.5) -- (8,-7.5)
(7.5,-4.5) -- (9,-7.5);
\draw
(8,-4.5) node [right] {$p_1$};

\filldraw
(10.5,-1.5) circle (0.1)
(11.5,-1.5) circle (0.1)
(12.5,-1.5) circle (0.1)
(13.5,-1.5) circle (0.1);
\draw
(14,-1.5) node [right] {$\cc_p$};

\filldraw
(10.5,1.5) circle (0.1)
(11.5,1.5) circle (0.1)
(12.5,1.5) circle (0.1)
(13.5,1.5) circle (0.1);
  \draw
(14,1.5) node [right] {$p_p(k)$};

\draw
(12, 1.5) node {$\ldots$}
(12, -1.5) node {$\ldots$};
    \draw
    (10.5,-1.5) -- (10.5,1.5)
    (10.5,-1.5) -- (11.5,1.5)
    (10.5,-1.5) -- (12.5,1.5)
    (10.5,-1.5) -- (13.5,1.5)
    (11.5,-1.5) -- (10.5,1.5)
    (11.5,-1.5) -- (11.5,1.5)
    (11.5,-1.5) -- (12.5,1.5)
    (11.5,-1.5) -- (13.5,1.5)
    (12.5,-1.5) -- (10.5,1.5)
    (12.5,-1.5) -- (11.5,1.5)
    (12.5,-1.5) -- (12.5,1.5)
    (12.5,-1.5) -- (13.5,1.5)
    (13.5,-1.5) -- (10.5,1.5)
    (13.5,-1.5) -- (11.5,1.5)
    (13.5,-1.5) -- (12.5,1.5)
    (13.5,-1.5) -- (13.5,1.5)
(10.5,-1.5) -- (7.5,-4.5)
      (11.5,-1.5) -- (7.5,-4.5)
      (12.5,-1.5) -- (7.5,-4.5)
    (13.5,-1.5) -- (7.5,-4.5)
(10.5,1.5)--(6,4.5)
    (10.5,1.5)--(7,4.5)
    (10.5,1.5)--(8,4.5)
    (10.5,1.5)--(9,4.5)
    (11.5,1.5)--(6,4.5)
    (11.5,1.5)--(7,4.5)
    (11.5,1.5)--(8,4.5)
    (11.5,1.5)--(9,4.5)
    (12.5,1.5)--(6,4.5)
    (12.5,1.5)--(7,4.5)
    (12.5,1.5)--(8,4.5)
    (12.5,1.5)--(9,4.5)
    (13.5,1.5)--(6,4.5)
    (13.5,1.5)--(7,4.5)
    (13.5,1.5)--(8,4.5)
    (13.5,1.5)--(9,4.5);

\filldraw
(6,4.5) circle (0.1)
(7,4.5) circle (0.1)
(8,4.5) circle (0.1)
(9,4.5) circle (0.1);
  \draw
(5.5,4.5) node [left] {$\cc_{2p}$};

\filldraw
(6,7.5) circle (0.1)
(7,7.5) circle (0.1)
(8,7.5) circle (0.1)
  (9,7.5) circle (0.1);
    \draw
(7.5,7.5) node [above] {$p_{2p}(k)$};

\draw
(7.5, 7.5) node {$\ldots$}
(7.5, 4.5) node {$\ldots$};
  \draw
  (6,4.5) -- (6,7.5)
  (6,4.5) -- (7,7.5)
  (6,4.5) -- (8,7.5)
  (6,4.5) -- (9,7.5)
  (7,4.5) -- (6,7.5)
  (7,4.5) -- (7,7.5)
  (7,4.5) -- (8,7.5)
  (7,4.5) -- (9,7.5)
  (8,4.5) -- (6,7.5)
  (8,4.5) -- (7,7.5)
  (8,4.5) -- (8,7.5)
  (8,4.5) -- (9,7.5)
  (9,4.5) -- (6,7.5)
  (9,4.5) -- (7,7.5)
  (9,4.5) -- (8,7.5)
  (9,4.5) -- (9,7.5)
(6,4.5) -- (7.5,-4.5)
  (7,4.5) -- (7.5,-4.5)
  (8,4.5) -- (7.5,-4.5)
  (9,4.5) -- (7.5,-4.5);
\filldraw
  (7.5,-4.5) circle (0.1);
\end{tikzpicture}
\caption{The boundary of $\fc_{L_{2p^2}}(\Gamma_{2p^2})^*$}\label{boundaryfp2}
\end{figure}
\begin{proof}
  It suffices to consider incidence relations as each of the other claims follow from Lemma \ref{2p2lineslem}, Proposition \ref{normformprop} or Theorem \ref{curvethm}. 
  Suppose $v \in E$ is primitive where $E \subset L$ is a primitive, totally isotropic sublattice of rank 2 and type $a$.
  By Proposition \ref{normformprop}, $E$ contains primitive  $v_1, v_2 \in L$ of  divisor $1$ and $a$, respectively.
  One verifies that $v^* \in \la v_1^*, v_2^* \ra$ and so, as $v_1^* \equiv 0 \bmod {L}$ then $v^* \equiv kv_2^* \bmod{ L}$ for some $k \in \ZZ$.
  Conversely, as
  \beq
  \bpm
  0 & 1 \\
  1 & k
  \epm \in \opn{GL}(2, \ZZ)
  \eeq
  then there exists a primitive isotropic vector $v \in E$ such that $v^* \equiv k v_2^* \bmod{ L}$ for all $k \in \ZZ$.
  By Lemma \ref{2p2lineslem}, there exists $v \in E$ with $v^*$ corresponding to $p_i$ or $p_i(k)$ if and only if $i \vert a$.
  The result then follows from Theorem \ref{bbdec}.
\end{proof}
\section{Acknowledgements}
Parts of this paper have their origins in my PhD thesis \cite{DawesThesis}: 
I thank Professor Gregory Sankaran for his supervision and the University of Bath for financial support in the form of a research studentship.
{\small
\bibliographystyle{alpha}
\bibliography{bib}{}
}
\bigskip
\texttt{matthew.r.dawes@bath.edu}
\end{document}